%% file: main.tex
\documentclass[11pt]{article}

\input{myarticle.sty}
\input{macrosetup}

\title{Counting thresholds for perfect matchings in hypergraphs}
\authDetails{Strahinja Gvozdi\'c}{sgvozdic@ethz.ch}{}[1]
\affil[1]{ETH Z\"urich\\Switzerland}
\date{}

\begin{document}
\maketitle

\input{abstract}

\input{intro}
\input{1deg}
\input{counting}
\input{approx-counting}
\input{concl}

\nocite{carnicer1999linear}
\bibliography{refs}

\appendix
\pagebreak
\input{append}

\end{document}

%% file: macrosetup.tex
\newcommand{\vw}{\mathbf{u}}

\newcommand{\fpmy}{\mathbf{y}}
\newcommand{\ctinf}{\tilde{\beta}_d(k)}

%% file: abstract.tex
\begin{abstract}
    In a $k$-uniform hypergraph, the minimum $d$-degree for some $0\le d\le k-1$ is the minimum number of edges containing any given $d$-set of vertices. An extension of the classical Dirac theorem guarantees that whenever the minimum $d$-degree of a $k$-uniform $n$-vertex hypergraph, $k\mid n$, is larger than a certain Dirac threshold, it contains at least one perfect matching. Moreover, it has been known for some time, due to Kwan, Safavi, and Wang \cite{kwan2026counting}, that for $d\ge k/2$ such hypergraphs contain not only one, but ``many'' perfect matchings, that is, at least as many as are expected in a random hypergraph with the same edge density. However, it has also been known that such a result could not be hoped for in general, as it already fails for $(d,k)=(1,3)$.

    In this paper we introduce new notions of the \emph{counting thresholds} and \emph{approximate counting thresholds}, above which a hypergraph is guaranteed to have at least this many perfect matchings. We show that these thresholds are well-defined and nontrivial for all $d,k,n$, that they are asymptotically related, and finally, we derive improved upper bounds by reducing to cases with smaller $d$ and $k$.
\end{abstract}

%% file: intro.tex
\section{Introduction}\label{sec-intro}

Dirac's theorem, commonly regarded as one of the fundamental results in graph theory, states that any graph $G$ on $n\ge3$ vertices with minimum degree at least $\frac{n}{2}$ contains a Hamilton cycle (i.e. a cycle that visits each vertex of $G$ exactly once). As this minimum degree condition already appears quite strong, a natural question arises: what is the minimum number of Hamilton cycles that can be guaranteed to exist in such a graph? This question was first raised in the literature by Bondy \cite{bondy1996basic}. Significant progress towards the correct answer was later made by S\'arkozy, Selkow and Szemer\'edi \cite{sarkozy2003number}, as an application of Szemer\'edi's celebrated regularity lemma \cite{szemeredi1975regular}. They showed that a graph satisfying Dirac's condition (we refer to such graphs simply as \emph{Dirac graphs} in the remainder of this paper) contains at least $c^nn!$ Hamilton cycles for some (small) constant $c>0$. They also conjectured that this constant can be improved to $1/2-o(1)$.

We first need to understand where this conjectured constant comes from. To that end, let $G(n,p)$ be the Erd\H{o}s--R\'enyi random graph for some $p$, that is, we include every edge among $n$ vertices independently with probability $p$. Then, by Chernoff bounds, it follows that with high probability its minimum degree is $(1-o(1))np$, and also with high probability -- which was shown by Janson \cite{janson1994numbers} -- the number of Hamilton cycles is $(1-o(1))\frac{1}{2}(n-1)!\,p^n=(1-o(1))^n\,n!\,p^n$ (i.e. is concentrated around the expected number). Thus, letting $p=\frac{1}{2}+\epsilon$ for $\epsilon>0$, we find that with high probability, $G$ is a Dirac graph and has at least $(\frac{1}{2}-o(1))^nn!$ Hamilton cycles, supporting the conjecture by S\'ark\"ozy, Selkow, and Szemer\'edi \cite{sarkozy2003number}. 

This intuition was later confirmed by Cuckler and Kahn \cite{cuckler2009entropy,cuckler2009hamiltonian}, who showed that any graph with minimum degree $d\geq\frac{n}{2}$ contains at least $\left(\frac{d}{e+o(1)}\right)^n$ Hamilton cycles, which is exactly what we would expect in a random graph with $p=\frac{d}{n}+\epsilon$, whose minimum degree is at least $d$ (with high probability). The random graph also certifies the tightness of this bound.

In this paper, however, we will focus on perfect matchings. Note that as long as the number of vertices of a Dirac graph is even, it produces at least one perfect matching, simply by taking every other edge of a Hamilton cycle. In fact, the results of Cuckler and Kahn \cite{cuckler2009entropy,cuckler2009hamiltonian} extend to perfect matchings: any graph with minimum degree $d\ge\frac{n}{2}$ (with $n$ even) has at least $\left(\frac{d}{e+o(1)}\right)^{n/2}$ perfect matchings, and this bound is tight.

Cuckler and Kahn essentially showed how the number of perfect matchings relates to the solution of a certain convex optimisation problem, which is easier to study in practice. Before we can state their results precisely, we need to make a few definitions.

\begin{definition}\label{def-fpm}
    Let $G=(V,E)$ be a $k$-uniform hypergraph. Then a weight assignment $\fpm\in[0,1]^E$ is a \emph{fractional perfect matching} of $G$ if for every vertex $v\in V$, it holds that $\sum_{e\ni v}\fpm_e=1$.
\end{definition}

\begin{definition}\label{def-entropy}
    Let $G=(V,E)$ be a $k$-uniform hypergraph and let $\fpm$ be its fractional perfect matching. Then the \emph{entropy} of $\fpm$ is given by
    $$h(\fpm):=-\sum_{e\in E}\fpm_e\ln \fpm_e,$$
    with the convention $0\ln0=0$. Moreover, the entropy of the hypergraph $G$ is defined as
    $$h(G):=\sup_{\fpm} h(\fpm),$$
    where $\fpm$ varies among all fractional perfect matchings of $G$.
\end{definition}

Let $\Phi(G)$ denote the number of perfect matchings in graph $G$. Now the main result of Cuckler and Kahn \cite{cuckler2009entropy,cuckler2009hamiltonian} is as follows.

\begin{theorem}[\cite{cuckler2009entropy,cuckler2009hamiltonian}]\label{thm-cuckler-kahn}
    Let $G$ be an $n$-vertex Dirac graph with $n$ even. Then
    $$\Phi(G)=e^{h(G)}\left(e^{-1}+o(1)\right)^{n/2}.$$
    Moreover, if $\delta(G)=d\ge\frac{n}{2}$, then $h(G)\ge \frac{n}{2}\ln d$, whence
    $$\Phi(G)\ge \Phi(K_n)(p+o(1))^{n/2},$$
    where $p=d/n$ and $K_n$ is a complete graph on $n$ vertices.
\end{theorem}

This problem of counting perfect matchings in Dirac graphs extends naturally to hypergraphs.

\begin{definition}\label{def-pm}
    Let $G$ be a $k$-uniform hypergraph. A \emph{perfect matching} in $G$ is a set of edges covering all vertices of $G$, with no two edges sharing a common vertex. We denote the number of perfect matchings in $G$ by $\Phi(G)$.
\end{definition}

\begin{definition}\label{def-d-deg}
    Let $G=(V,E)$ be a $k$-uniform hypergraph and let $d$ be a positive integer with $0\le d\le k-1$. Given a $d$-subset $A=\{v_1,\dots,v_d\}$ of $V$, we denote by $\deg_dA$ its \emph{$d$-degree}, that is, the number of edges in $E$ that contain $A$ as a subset. We also define $\delta_d(G):=\min_{A\in\binom{V}{d}}\deg A$ the minimum $d$-degree among all $d$-subsets of $V$.
\end{definition}

\begin{definition}\label{thres-pm}
    Let $n$ and $k$ be positive integers with $k\mid n$, and let $d$ be an integer with $1\le d\le k-1$. Then $m_d(k,n)$ denotes the $d$-degree \emph{Dirac threshold} of $n$-vertex $k$-uniform hypergraphs, that is, the minimum value of $\delta$ such that $\Phi(G)\geq 1$ whenever $\delta_d(G)\ge \delta$. Moreover, let
    $$\alpha_d(k):=\lim_{\substack{n\to\infty\\k\mid n}} \frac{m_d(k,n)}{\binom{n-d}{k-d}}$$
    be the \emph{asymptotic Dirac threshold} for $d$-degree in $k$-uniform hypergraphs.
\end{definition}

Ferber and Kwan \cite{ferber2022dirac} showed that the limit in the previous theorem exists, thereby justifying the definition. It has been conjectured that the exact values of (asymptotic) Dirac thresholds arise from only two extremal families of hypergraphs (see, for example, \cite{alon2012large} for a discussion of these constructions).


\begin{conjecture}[\cite{han2009perfect,kuhn2009embedding}]\label{conj-alpha}
    For all positive integers $k,d$ with $d\le k-1$,
    $$\alpha_d(k)=\max\left\{\frac{1}{2},\,1-\left(1-\frac{1}{k}\right)^{k-d}\right\}.$$
\end{conjecture}

Very recently, several independent groups of authors \cite{fu2026sharp,nie2026feige,stander2026conditional} announced proofs of the long-\linebreak standing Feige's conjecture. This conjecture is itself a special case of the conjecture of Samuels, whose relationship with the problem of determining Dirac thresholds of hypergraphs has been studied in \cite{alon2012large}. It follows from the results of this paper that Conjecture \ref{conj-alpha} holds assuming Feige's conjecture. Before these breakthrough proofs of Feige's conjecture, only sporadic exact values of (asymptotic) Dirac thresholds were known.

Even without knowing these thresholds precisely, it is still possible to count perfect matchings in Dirac hypergraphs. Since in these results it is usually required that the asymptotic Dirac threshold is strictly surpassed, we make the following definition.

\begin{definition}\label{def-gamma-dirac}
    Let $n$, $k$, and $d$ be positive integers, $k\mid n$, $1\le d\le k-1$, and let $\gamma>0$. A $k$-uniform $n$-vertex hypergraph is said to be \emph{$(d,\gamma)$-Dirac} (or simply $\gamma$-Dirac if $d$ is understood from the context) if $\delta_d(G)\ge(\alpha_d(k)+\gamma)\binom{n-d}{k-d}$.
\end{definition}

First result in this direction is due to Ferber, Krivelevich, and Sudakov \cite{ferber2014counting} who showed that a $(k-1,\gamma)$-Dirac $k$-uniform hypergraph $G$ satisfies
$$\Phi(G)\ge \Phi(\multicomp kn)\left(\frac{\delta_{k-1}(G)}{n}+o(1)\right)^{n/k},$$
where $\multicomp kn$ denotes the complete $k$-uniform hypergraph on $n$ vertices. This was later extended by Kang, Kelly, K\"uhn, Osthus, and Pfenninger \cite{kang2024perfect}, and independently Pham, Sah, Sawhney, and Simkin \cite{pham2022toolkit} to all $1\le d\le k-1$, but with larger error term: they showed that in a $(d,\gamma)$-Dirac $k$-uniform $n$-vertex hypergraph, the number of perfect matchings is at least $n^{(1-1/k)n}\exp(-O(n))$.

Finally, the hypergraph analogue of the results of Cuckler and Kahn was recently achieved in its full strength due to work of Kwan, Safavi, and Wang \cite{kwan2026counting}.

\begin{theorem}[\cite{kwan2026counting}]\label{thm-kwan}
    Let $\gamma>0$ be a real constant, and let $n,k,d$ be integers with $k\mid n$, $1\le d\le k-1$. Then every $(d,\gamma)$-Dirac $k$-uniform $n$-vertex hypergraph $G$ satisfies
    \begin{equation}\label{eq-phi-entropy}
        \Phi(G)=e^{h(G)}\left(e^{-1}+o(1)\right)^{\left(1-\frac{1}{k}\right)n}.
    \end{equation}
    Moreover, if $d\ge \frac{k}{2}$, then
    \begin{equation}\label{eq-h-ge-random}
        h(G)\ge\frac{n}{k}\ln\left(\frac{k}{n}\,\frac{\binom{n}{d}}{\binom{k}{d}}\,\delta_d(G)\right),
    \end{equation}
    whence the number of perfect matchings satisfies
    \begin{equation}\label{eq-phi-ge-random}
        \Phi(G)\ge\Phi(K^k_n)\,(p+o(1))^{n/k},
    \end{equation}
    where $p=\delta_d(G)/\binom{n-d}{k-d}$.
\end{theorem}

Although at first sight one could hope to extend the results Kwan, Safavi, and Wang to all $1\le d\le k-1$, it was observed by Sauermann in \cite[Section 5]{ferber2023counting} that already for $(d,k)=(1,3)$ this is doomed to fail. We will discuss Sauermann's construction in more detail in Section \ref{sec-concl}. The main contribution of this paper is the following theorem which shows that for any $d$, as long as the minimum $d$-degree is large enough, one can still find at least as many perfect matchings as is expected in a random hypergraph with the same edge density.

\begin{theorem}\label{thm-main}
    Let $k\ge2$ and $1\le d< k-1$ be positive integers. If $G$ is a $k$-uniform hypergraph that satisfies $\delta_d(G)> \frac{k-1}{k}\frac{n}{n-1}\binom{n-d}{k-d}$, then
    $$h(G)\ge\frac{n}{k}\ln\left(\frac{k}{n}\,\frac{\binom{n}{d}}{\binom{k}{d}}\,\delta_d(G)\right).$$
\end{theorem}

Asymptotically, this theorem shows that we get the desired number of perfect matchings whenever the minimum $d$-degree of $G$ deviates from the theoretical maximum by no more than a $1/k$-fraction. The question then arises to determine the exact thresholds above which entropy must satisfy (\ref{eq-h-ge-random}), and we call these thresholds \emph{counting thresholds} and denote $c_d(k,n)$. It follows from Theorem \ref{thm-main} that they exist and are nontrivial.

Unlike the case of Dirac thresholds, here we are not able to show that the ratio $c_d(k,n)/\binom{n-d}{k-d}$ tends to a limit as $n$ goes to infinity. Nevertheless, we can still give a result that is almost equally good in practice (i.e. after an application of Theorem \ref{thm-main}): whenever the minimum $d$-degree of a $k$-uniform $n$-vertex graph (for $n$ large enough) is asymptotically larger than $\liminf_{n\to\infty}c_d(k,n)/\binom{n-d}{k-d}$, its entropy satisfies (\ref{eq-h-ge-random}) up to an error term of order $o(n)$. This result justifies introducing the notion of \emph{approximate counting thresholds} $\beta_d(k)$ to compensate for the error term.

Finally, we are able to give slightly better bounds on approximate counting thresholds by reducing to the cases with smaller $d$ and $k$. Namely, we show that $\beta_d(k)$ is upper-bounded by $\beta_1(k/d)$ whenever $d\mid k$, and by $\beta_{d-l}(k-l)$ for all $l<d$.

Note that in the statement of Theorem \ref{thm-main} we do not require that $k\mid n$. In fact, it is crucial for us that the notions of fractional perfect matchings and entropy still make sense even if no perfect matching exists due to divisibility constraints.

\subsection{Structure of the paper}\label{sec-struct}

After a short exposition of preliminary results in Section \ref{sec-preliminaries}, in Section \ref{sec-1-deg-conditions} we show Theorem \ref{thm-main}. In Section \ref{sec-count} we introduce counting and approximate counting thresholds, and show how they relate. In Section \ref{sec-approx-count} we give some general upper bounds on approximate counting thresholds. Finally, in Section \ref{sec-concl} we present some concluding remarks and possibilities for further development. For completeness, we also include in Appendix \ref{sec-append} a proof due to Hoffman of the important Theorem \ref{thm-pmd}.

\subsection{Preliminaries}\label{sec-preliminaries}




It is worth noting that a lower bound on $\delta_{d'}(G)$ for some $d'$ implies a bound on $\delta_d(G)$ for all $d\le d'$ via a double counting argument.

\begin{proposition}\label{prop-mono-delta}
    Let $G$ be a $k$-uniform $n$-vertex graph. Then
    $$\frac{\delta_0(G)}{\binom{n}{k}}\ge\cdots\ge\frac{\delta(G)}{\binom{n-k+1}{1}}.$$
    Consequently, it holds that $\alpha_0(k)\ge\alpha_1(k)\ge\cdots\ge\alpha_{k-1}(k)$ for all $k$.
\end{proposition}

\begin{proof}
    Let $0\le d< d'\le k-1$, and let $S\in \binom{V(G)}{d}$ be an arbitrary $d$-set of vertices. Then
    \begin{align*}
        \deg_dS &\,\,\,=\,\,\, \binom{k-d}{d'-d}\deg_dS\Big/\binom{k-d}{d'-d} \,\,\,=\,\,\, \frac{1}{\binom{k-d}{d'-d}}\sum_{S\subseteq e\in E(G)} \sum_{A\in\binom{e\setminus S}{d'-d}}1 \\
        &\,\,\,=\,\,\, \frac{1}{\binom{k-d}{d'-d}}\sum_{A\in\binom{V(G)\setminus S}{d'-d}}\sum_{S\cup A\subseteq e\in E(G)}1 \,\,\,=\,\,\, \frac{1}{\binom{k-d}{d'-d}}\sum_{A\in\binom{V(G)\setminus S}{d'-d}}\deg_{d'}(S\cup A) \\
        &\,\,\,\ge\,\,\, \frac{1}{\binom{k-d}{d'-d}}\sum_{A\in\binom{V(G)\setminus S}{d'-d}}\delta_{d'}(G) \,\,\,=\,\,\, \binom{n-d}{d'-d}\delta_{d'}(G)\Big/\binom{k-d}{d'-d}.
    \end{align*}
    Since $\binom{n-d}{d'-d}/\binom{k-d}{d'-d}=\binom{n-d}{k-d}/\binom{n-d'}{k-d'}$ and $S$ was arbitrary, the proposition follows.
\end{proof}

As we will see in Section \ref{sec-1-deg-conditions}, it is desirable to find a set of conditions on the entries of real square matrices that guarantee that the row or column sums of the inverse matrix are positive (or nonnegative). The following definition gives one such set of conditions, which was originally discovered by Hoffman \cite{hoffman1965nonsingularity} in his work on the nonsingularity of real matrices.

\begin{definition}\label{def-pmd}
    An $n\times n$ real matrix $A=(a_{ij})$ is said to be \emph{positive mean-dominant} or PMD if
    \begin{align}
        \sum_{j=1}^n a_{ij}\ge 0\hspace{0.5cm}&\text{for } i=1,\dots,n\label{eq-pmd1} \\
        \frac{1}{n}\sum_{j=1}^n a_{ij}\ge a_{ik} \hspace{0.5cm} &\text{for } i=1,\dots,n,\,k=1,\dots,n, \,k\neq i \label{eq-pmd2}
    \end{align}
    If (\ref{eq-pmd1}) and (\ref{eq-pmd2}) are both strict, $A$ is said to be \emph{strictly positive mean-dominant}.
\end{definition}

(Strictly) positive mean-dominant matrices are referred to as $B$- and $B_0$-matrices, respectively, by some authors (e.g. \cite{christensen2019comparative,pena2001class}). The following theorem gives the key properties of PMD matrices that we will need in this paper.

\begin{theorem}[\cite{hoffman1965nonsingularity}]\label{thm-pmd}
    Let $A$ be an $n\times n$ real PMD matrix. Then $\det A\ge 0$, with inequality strict if $A$ is strictly PMD. If moreover $A$ is invertible (which is always the case for $A$ strictly PMD), we have $\sum_{i=1}^n a^{-1}_{ij}\ge0$ for all $j=1,\dots,n$, where $A^{-1}=(a^{-1})_{ij}$.
\end{theorem}

We include the proof of this theorem in Appendix \ref{sec-append} for completeness.

%% file: 1deg.tex
\section{1-degree conditions for many perfect matchings}\label{sec-1-deg-conditions}

In this section, we show a (nontrivial) $1$-degree condition that ensures a Dirac hypergraph has many perfect matchings -- more precisely, that it satisfies (\ref{eq-phi-ge-random}). This condition will, in turn, imply Theorem \ref{thm-main} via Proposition \ref{prop-mono-delta}.

For completeness, we show how the lower bound on entropy in (\ref{eq-h-ge-random}) implies the lower bound on the number of perfect matchings (\ref{eq-phi-ge-random}) whenever (\ref{eq-phi-entropy}) holds, which is the case for all $(d,\gamma)$-Dirac hypergraphs by Theorem \ref{thm-kwan}. Note that
$$\ln\Phi(\multicomp kn) =\ln \frac{n!}{(n/k)!(k!)^{n/k}}=n\ln n - \frac{n}{k}\ln\frac{n}{k}-\left(1-\frac{1}{k}\right)n-\frac{n}{k}\ln(k!)+o(n)$$
by Stirling's approximation, so the logarithm of the right-hand side of (\ref{eq-phi-ge-random}) equals
$$\ln\Phi(\multicomp kn)+\frac{n}{k}\ln p+o(n)=n\ln n+\frac{n}{k}\ln\left(\frac{k}{n}p\right)-\left(1-\frac{1}{k}\right)n -\frac{n}{k}\ln(k!)+o(n).$$
On the other hand, supposing (\ref{eq-phi-entropy}) and (\ref{eq-h-ge-random}) hold and recalling that $p=\delta_d(G)/\binom{n-d}{k-d}$, the logarithm of the left-hand side satisfies
\begin{align*} 
    \ln\Phi(G) &\ge h(G)-\left(1-\frac{1}{k }\right)n+o(n) \\
    &\ge \frac{n}{k}\ln\left(\frac{k }{n }\frac{\binom{n }{d}}{\binom{k}{d}}\binom{n-d}{k-d}p\right)-\left(1-\frac{1}{k }\right)n+o(n) \\
    &=\frac{n }{k }\ln\left(\frac{k }{n }\binom{n}{k}p\right)-\left(1-\frac{1}{k }\right)n+o(n) \\
    &= \frac{n }{k }\ln\left(\frac{k }{n }p\right)+n\ln n -\frac{n}{k}\ln(k!)-\left(1-\frac{1}{k }\right)n+o(n),
\end{align*}
completing the proof.

Therefore, in order to acheive the desired lower bound on the number of perfect matchings in a $(d,\gamma)$-Dirac hypergraph, it suffices to show that (\ref{eq-h-ge-random}) holds. As announced, the following theorem does this whenever $1$-degree is large enough.

\begin{theorem}\label{thm-1-deg}
    Let $k\ge2$ be an integer. If $G=(V,E)$ is a $k$-uniform hypergraph that satisfies $\delta_1(G)> \frac{n}{k}\binom{n-2}{k-2}$, then $h(G)\ge \frac{n}{k}\ln{\delta_1(G)}$.
\end{theorem}

Our proof follows the ideas of Cuckler and Kahn \cite[Theorem 1.3]{cuckler2009hamiltonian}. However, we use the theory of PMD matrices to give an easier proof of \cite[(24)]{cuckler2009hamiltonian} and generalise it to higher uniformities.

\begin{proof}
    Suppose $\fpm\in\R^{E}$ is a fractional perfect matching of $G$. Note that
    \begin{equation}\label{eq-pm-sum}
        \one_{E}\trans\fpm=\sum_{e\in E}\fpm_e=\frac{1}{k}\sum_{e\in E}\sum_{v\in e}\fpm_{e}=\frac{1}{k}\sum_{v\in V}\sum_{e\ni v} \fpm_e=\frac{1}{k}\sum_{v\in V}1=\frac{n}{k},
    \end{equation}
    so
    $$h(G)\ge h(\fpm)=-\sum_{e\in E}\fpm_e\ln\fpm_e=-\frac{n}{k}\sum_{e\in E}\frac{k\fpm_e}{n}\ln\fpm_e\ge -\frac{n}{k}\ln\frac{k}{n}\sum_{e\in E}\fpm_e^2$$
    by Jensen's inequality applied to the concave function $x\mapsto\ln x$. Therefore, it suffices to find a fractional perfect matching $\fpm$ of $G$ that satisfies $\fpm\trans\fpm=\sum_{e\in E}\fpm_e^2\le \frac{n}{k\delta_1(G)}$.

    Now let $M\in\mathcal{M}_{V\times E}(\R)$ be the vertex-edge incidence matrix of $G$. We can rewrite the condition for $\fpm$ being a fractional perfect matching as $\fpm\ge0$ and $M\fpm=\one_V$. Moreover, given a positive vertex-weight assignment $\vw\in\R_{\ge0}^V$, $M\trans\vw$ gives a positive edge-weight assignment (since $M\ge0$), which becomes a fractional perfect matching if it satisfies $\one_V=M(M\trans\vw)=MM\trans\vw$ as well.
    
    In the graph case ($k=2$), matrix $(a_{vw})_{v,w\in V}=A:=MM\trans$ is well-studied and known as the signless Laplacian matrix of $G$. Here, it will suffice to note that $a_{vv}=\deg_1(v)$ and $a_{vw}=\deg_2(v,w)$ for $v\ne w$. We check that $A$ is strictly PMD under the hypotheses of the theorem: all entries are nonnegative and diagonal entries are strictly positive, so the row sums are strictly positive as well; as for the mean dominance, we have for every $v\in V$ and $w\ne v$,
    \begin{align*}
        \frac{1}{n}\sum_{s\in V}a_{vs}&=\frac{1}{n}\left(\deg_1(v)+\sum_{s\ne v}\deg_2(v,s)\right)=\frac{k}{n}\deg_1(v)\\&>\binom{n-2}{k-2}\ge\deg_2(v,w)=a_{vw},
    \end{align*}
    which is exactly what we wanted.

    By Theorem \ref{thm-pmd}, $A$ is invertible and $A^{-1}$ has nonnegative column sums. That is, $\one_V\trans A^{-1}\ge0$. But $A$ is symmetric, so this implies that $\vw:=A^{-1}\one_V=\left(\one_V\trans A^{-1}\right)\trans\ge0$. Therefore, $\fpm:=M\trans\vw$ is a well-defined fractional perfect matching.

    We only need to check that $\fpm$ satisfies $\fpm\trans\fpm\le\frac{n}{k\delta_1(G)}$. We have:
    \begin{align*}
        \fpm\trans\fpm&=\vw\trans MM\trans\vw=\vw\trans A\vw=\vw\trans\one_V=\sum_{v\in V}\vw_v\\&=\frac{1}{\delta_1(G)}\sum_{v\in V}\delta_1(G)\vw_v\le\frac{1}{\delta_1(G)}\sum_{v\in V}\deg_1(v)\vw_v\\&=\frac{1}{\delta_1(G)}\sum_{v\in V}\sum_{e\ni v}\vw_v=\frac{1}{\delta_1(G)}\sum_{e\in E}\sum_{v\in e}\vw_v\\&=\frac{1}{\delta_1(G)}\sum_{e\in E}\fpm_e=\frac{n}{k\delta_1(G)},
    \end{align*}
    so we are done.
\end{proof}

\begin{remark}
    The minimum $1$-degree condition in Theorem \ref{thm-1-deg} is asymptotically of the form $1-1/k$, which strictly dominates $\alpha_1(k)$ for $1\le d<k/2$, owing to the results of K\"uhn, Osthus, and Townsend \cite{kuhn2014fractional} who showed that $\alpha_d(k)\le 1-\frac{d}{k}-\frac{k-d-1}{k^{k-d}}$. Thus, hypergraphs that satisfy the conditions of Theorem \ref{thm-1-deg} are $(1,\gamma)$-Dirac for sufficiently small $\gamma>0$, and therefore satisfy (\ref{eq-phi-entropy}).
\end{remark}

\begin{remark}
    Note that in the previous proof the strict inequality in the minimum degree condition $\delta_1(G)>\frac{n}{k}\binom{n-2}{k-2}$ implies that $A$ is \emph{strictly} PMD and thus invertible. Thus the proof could hold even when $\delta_1(G)=\frac{n}{k}\binom{n-2}{k-2}$ if one can show that $A$ is still invertible.
    
    In fact, if the signless Laplacian matrix of a $k$-uniform hypergraph $G$ is singular, then $G$ must be a \emph{partially bipartite hypergraph}, that is, a hypergraph with three vertex parts $P$, $N$, and $Z$ and with all edges intersecting both $P$ and $N$ or being fully contained in $Z$. Calculations show that for a $k$-uniform $(1,\gamma)$-Dirac partially bipartite hypergraph $G$ there can be no vertices in part $Z$, since they would have too small $1$-degree. 
    
    Moreover, for $k=2$ the only Dirac bipartite graph is $K_{n/2,n/2}$, which can be treated separately (see Cuckler and Kahn \cite[Lemma 3.3]{cuckler2009hamiltonian}). For $k=3$, one can show that even the $(1,\gamma)$-Dirac bipartite hypergraphs have $A$ invertible, so that Theorem \ref{thm-1-deg} holds in the slightly stronger form with $\delta_1(G)\ge\frac{n}{k}\binom{n-2}{k-2}$. 
\end{remark}

Theorem \ref{thm-main} is now an easy corollary by Proposition \ref{prop-mono-delta}.

\begin{proof}[Proof of Theorem \ref{thm-main}]
    From $\delta_d(G)> \frac{k-1}{k}\frac{n}{n-1}\binom{n-d}{k-d}$ it follows that
    $$\delta_1(G)\ge \frac{\binom{n-1}{k-1}}{\binom{n-d}{k-d}}\delta_d(G) > \frac{n}{k}\binom{n-2}{k-2}$$
    by Proposition \ref{prop-mono-delta}, so the result follows immediately from Theorem \ref{thm-1-deg}.
\end{proof}

%% file: counting.tex
\section{Counting thresholds}\label{sec-count}

The results of the previous section lead to the following definition.

\begin{definition}\label{def-count}
    Let $n,k,d$ be positive integers, $1\le d\le k-1$. Then the $d$-degree \emph{counting threshold} for $n$ and $k$, denoted $c_d(k,n)$, is the minimum value of $\delta$ such that all $k$-uniform hypergraphs $G$ on $n$ vertices with $\delta_d(G)\ge \delta$ satisfy (\ref{eq-h-ge-random}), that is,
    $$h(G)\ge\frac{n}{k}\ln\left(\frac{k}{n}\,\frac{\binom{n}{d}}{\binom{k}{d}}\,\delta_d(G)\right).$$
\end{definition}

Moreover, Theorem \ref{thm-main} asserts that $c_d(k,n)\le(1-\frac{1}{k })\frac{n }{n-1}\binom{n-d}{k-d}+1$, so counting thresholds are always bounded away from $\binom{n-d }{k-d}$ for fixed $d$ and $k$. On the other hand, this definition is ``too precise'' for the applications we need it for: every application of Theorem \ref{thm-kwan} gives an error term of $o(n)$ in the exponent, so we would not lose much by permitting the same error term in the definition of the counting threshold. This is the goal of the following definition.

\begin{definition}\label{def-approx-count}
    Let $1\le d\le k-1$ be positive integers. Then the $d$-degree $k$-uniform \emph{approximate counting threshold}, denoted $\beta_d(k)$, is the minimum $\beta>0$ such that for all $\gamma,\epsilon>0$, there exists $n_0$ such that for all $n\ge n_0$, if $G$ is a $k$-uniform $n$-vertex hypergraph with $\delta_d(G)\ge(\beta+\gamma)\binom{n-d}{k-d}$, then it satisfies
    $$h(G)\ge \frac{n}{k}\ln\left((1-\epsilon)\frac{k}{n}\,\frac{\binom{n}{d}}{\binom{k}{d}}\,\delta_d(G)\right).$$
\end{definition}

Again, this definition makes sense, as $\beta$ is bounded away from $1$ by $1/k$ by Theorem \ref{thm-main}. Moreover, $\beta$ is clearly upper-bounded by $\limsup_{n\to\infty} c_d(k,n)/\binom{n-d}{k-d}$. It is less clear that it is also upper-bounded by $\liminf_{n\to\infty} c_d(k,n)/\binom{n-d}{k-d}$. We dedicate the remainder of this section to the proof of this fact, for which we will need a few lemmas. The first one lets us combine fractional perfect matchings of subhypergraphs to obtain a fractional perfect matching of the original hypergraph, with the entropy bounded in terms of the entropies of subhypergraphs. Its proof relies in part on the methods of Kwan, Safavi, and Wang in \cite{kwan2026counting}, which will be exploited even further in the next section.

\begin{lemma}\label{le-ent-comb}
    Let $k,q,n$ be positive integers. Let $V$ be a set of $n$ vertices, and let $G$ and $H$ be a $k$-uniform hypergraph and a $q$-uniform hypergraph on vertex set $V$, respectively. Let $\fpmy:E(H)\to\Rpos$ be a fractional perfect matching of $H$, and for each $f\in E(H)$, let $\fpm^f:E(G[f])\to\Rpos$ be a fractional perfect matching of the subgraph of $G$ induced by $f$. Then
    \begin{equation*}
        \fpm:E(G)\to\Rpos, \hspace{7pt} e\mapsto\sum_{e\subseteq f\in E(H)}\fpmy_f \fpm^f_e
    \end{equation*}
    defines a fractional perfect matching of $G$ that satisfies
    \begin{equation*}
        h(\fpm)\ge \sum_{f\in E(H)}\fpmy_f h(\fpm^f) + \frac{q }{k }h(\fpmy) - \frac{n }{k }\ln \binom{n-k}{q-k}.
    \end{equation*}
\end{lemma}

\begin{proof}
    First we check that $\fpm$ defines a fractional perfect matching of $G$. It is clear that it is nonnegative, so take any $v\in V$ and note that
    $$\sum_{v\in e\in E(G)}\fpm_e=\sum_{v\in e\in E(G)}\sum_{e\subseteq f\in E(H)} \fpmy_f\fpm^f_e=\sum_{v\in f\in E(H)}\fpmy_f\sum_{v\in e\in E(G[f])}\fpm^f_e = \sum_{v\in f\in E(H)}\fpmy_f=1.$$
    So $\fpm$ really is a fractional perfect matching.

    Next we show the lower bound on the entropy of $\fpm$. For $e\in E(G)$, let $n_e$ denote the number of edges of $H$ containing $e$. By definition, we have:
    \begin{align*} 
        h(\fpm) &=\sum_{e\in E(G)}\fpm_e\ln\frac{1}{\fpm_e} = \sum_{e\in E(G)}\fpm_e\ln\frac{1}{n_e }\frac{n_e }{\fpm_e} \\
        &= -\sum_{e\in E(G)}(\ln n_e)\fpm_e + \sum_{e\in E(G)}n_e\frac{\fpm_e}{n_e}\ln\frac{n_e }{\fpm_e} \\
        &\ge -\ln\binom{n-k}{q-k}\sum_{e\in E(G)}\fpm_e +\sum_{e\in E(G)}n_e\cdot\frac{\sum_{f\supseteq e}\fpmy_f\fpm^f_e}{n_e}\ln\frac{n_e }{\sum_{f\supseteq e}\fpmy_f\fpm^f_e} \\
        &\ge -\frac{n }{k }\ln\binom{n-k}{q-k} + \sum_{e\in E(G)}n_e\sum_{f\supseteq e}\frac{1}{n_e }\fpmy_f\fpm^f_e\ln\frac{1}{\fpmy_f\fpm^f_e} \\
        &= -\frac{n }{k }\ln\binom{n-k}{q-k} + \sum_{e\in E(G)}\sum_{f\supseteq e}\fpm^f_e\fpmy_f\ln\frac{1}{\fpmy_f} + \sum_{e\in E(G)}\sum_{f\supseteq e}\fpmy_f\fpm^f_e\ln\frac{1}{\fpm^f_e} \\
        &= -\frac{n }{k }\ln\binom{n-k}{q-k} + \sum_{f\in E(H)}\left(\fpmy_f\ln\frac{1}{\fpmy_f}\right)\sum_{e\in E(G[f])}\fpm^f_e + \sum_{f\in E(H)}\fpmy_f\sum_{e\in E(G[f])}\fpm^f_e\ln\frac{1}{\fpm^f_e} \\
        &= -\frac{n }{k }\ln\binom{n-k}{q-k} + \frac{q }{k }\sum_{f\in E(H)}\fpmy_f\ln\frac{1}{\fpmy_f} + \sum_{f\in E(H)}\fpmy_f h(\fpm^f) \\
        &= -\frac{n }{k }\ln\binom{n-k}{q-k} + \frac{q }{k }h(\fpmy) + \sum_{f\in E(H)}\fpmy_f h(\fpm^f),
    \end{align*}
    where for the first inequality we used that $n_e\le \binom{n-k}{q-k}$ for all $e\in E(G)$, for the second we used (\ref{eq-pm-sum}) and applied Jensen's inequality to the concave function $x\mapsto x\ln\frac{1}{x}$, and for the fifth equality we again used (\ref{eq-pm-sum}). This completes the proof.
\end{proof}

The following lemma is essentially that of Ferber and Kwan \cite{ferber2023counting}, adapted to treat $1$-degree instead of $0$-degree.

\begin{lemma}\label{le-rooted}
    Let $1\le d\le k-1$ be integers, and let $\eta>0$ be a real constant. There exists $c=c(k)>0$ such that for all sufficiently large $q$ and $n$, the following holds.

    Let $G$ be a $k$-uniform $n$-vertex hypergraph satisfying $\delta_d(G)\ge p\binom{n-d}{k-d}$ for some $p\in[0,1]$, and let $v\in V(G)$ be a fixed vertex. Let $S$ be a $q$-subset of $V(G)$ chosen uniformly among all subsets that contain $v$. Then with probability at least $1-\binom{q}{d}e^{-c\eta^2 q}$, the subgraph of $G$ induced by $S$ has minimum $d$-degree at least $\delta_d(G[S])\ge(p-\eta)\binom{q-d}{k-d}$.
\end{lemma}

\begin{proof}
    Let $d,k,\eta$ be as in the statement of the lemma. We will choose $c$ implicitly later. Suppose also that $n$ and $q$ are large.
    
    For a $d$-set $A\in\binom{V(G)}{d}$, let $X_A$ denote $\deg_d^{G[S]}A$ (i.e. the $d$-degree of $A$ in the induced subgraph $G[S]$) conditioned on $S\supseteq A$. We distinguish two cases: $A\ni v$ and $A\not\ni v$.

    For the first case, $d$ members of $S$ are already chosen, and the remaining $q-d$ vertices are chosen uniformly from the set of $n-d$ total vertices. Therefore, each of at least $p\binom{n-d}{k-d}$ edges of $G$ containing $A$ is found in $S$ with probability $\binom{n-k}{q-k}/\binom{n-d}{q-d}=\binom{q-d}{k-d}/\binom{n-d}{k-d}$, so by linearity of expectation $\expect X_A\ge p\binom{q-d}{k-d}$.

    On the other hand, when $A\not\ni v$, there are two types of edges we are concerned about: those that contain $v$ in addition to $A$ and those that do not. However, we will simply ignore the first type to keep the calculations simple. To this end, note that there are at most $\binom{n-d-1}{k-d-1}$ edges of $G$ that contain $A\cup\{v\}$, so there are at least $p\binom{n-d}{k-d}-\binom{n-d-1}{k-d-1}$ edges that contain $A$, but not $v$. Since the remaining vertices of $S$ on top of $A$ and $v$ are chosen uniformly from a set of $n-d-1$ possible vertices, each of these edges appears in $S$ with probability $\binom{q-d-1}{k-d}/\binom{n-d-1}{k-d}$, so that
    $$\expect X_A\ge \left(p\binom{n-d}{k-d}-\binom{n-d-1}{k-d-1}\right)\frac{\binom{q-d-1}{k-d}}{\binom{n-d-1}{k-d}}=\frac{p(n-d)-(k-d)}{n-k }\frac{q-k }{q-d }\binom{q-d}{k-d}\ge \left(p-\frac{\eta}{2}\right)\binom{q-d}{k-d}$$
    for $q$ and $n$ large. So in both cases we got $\expect X_A\ge\left(p-\frac{\eta}{2}\right)\binom{q-d}{k-d}$.

    Now, given $A$, the remaining $q-d$ or $q-d-1$ vertices of $S$ are chosen uniformly at random. If we expose these random vertices one at a time, we note that replacing one vertex by another can change $X_A$ by at most $\binom{q}{k-d-1}$. Then an application of the Azuma-Hoeffding inequality gives
    $$\prob\left[X_A\le (p-\eta)\binom{q-d}{k-d}\right]\le \exp\left(-\frac{(\eta/2)^2\binom{q-d}{k-d}^2}{2\binom{q}{k-d-1}^2(q-d)}\right)\le e^{-c\eta^2q}$$
    for small enough $c=c(k)>0$. Since this holds for arbitrary $A\in\binom{V(G)}{d}$, by applying a union bound to all $\binom{q}{d}$ $d$-subsets of $S$ we conclude.
\end{proof}

For integers $1\le d\le k-1$, write $\ctinf:= \liminf_{n\to\infty} c_d(k,n)/\binom{n-d}{k-d}$. We are now ready to prove the promised fact.

\begin{theorem}
    Let $1\le d\le k-1$ be integers. Then for every $\gamma,\epsilon>0$, there exists $n_0$ such that for all $n\ge n_0$, the following holds. If $G$ is a $k$-uniform $n$-vertex hypergraph with $\delta_d(G)\ge (\ctinf+\gamma)\binom{n-d}{k-d}$, then the entropy of $G$ satisfies
    $$h(G)\ge \frac{n}{k}\ln\left((1-\epsilon)\frac{k}{n}\,\frac{\binom{n}{d}}{\binom{k}{d}}\,\delta_d(G)\right).$$
    In other words, $\beta_d(k)\le\ctinf$.
\end{theorem}

\begin{proof}
    Let $G$ be as in the statement of the theorem and let $p=\delta_d(G)/\binom{n-d}{k-d}\ge \ctinf+\gamma$. By definition of $\ctinf$, there are arbitrarily large integers $q$ divisible by $k$ such that $c_d(k,q)\le (\ctinf+\gamma/2)\binom{q-d}{k-d}$.
    
    Fix one such $q$, and call a $q$-set $S$ of vertices of $V(G)$ ``good'' if $\delta_d(G[S])\ge (p-\eta)\binom{q-d}{k-d}$ for some $0<\eta<\gamma/2$. Let $H$ be a $q$-uniform hypergraph on $V(G)$ whose edges are all good $q$-sets of $G$. Then by Lemma \ref{le-rooted}, $\delta_1(H)\ge (1-\rho)\binom{n-1}{q-1}$ for $\rho=\rho(\eta,q):=\binom{q}{d}e^{-c\eta^2q}$ and some constant $c>0$. Since $\rho$ goes to $0$ exponentially fast as $q$ grows to infinity with $\eta$ fixed, for $q$ sufficiently large it holds that $\rho<\frac{1}{q}\frac{n}{n-1}-\frac{1}{n-1}$, so by Theorem \ref{thm-1-deg} the entropy of $H$ is at least $h(H)\ge \frac{n }{q }\ln\delta_1(H)$, achieved by some fractional perfect matching $\fpmy:E(H)\to\Rpos$.

    By definition of the hypergraph $H$, every edge $f\in E(H)$ satisfies $\delta_d(G[f])\ge (\ctinf+\gamma/2)\binom{q-d}{k-d}\ge c_d(k,q)$, so there exists a fractional perfect matching $\fpm^f$ of $G[f]$ with
    $$h(\fpm^f)\ge \frac{q}{k }\ln\left(\frac{k }{q}\frac{\binom{q}{d}}{\binom{k}{d }}\delta_d(G[f])\right)\ge \frac{q}{k }\ln\left((p-\eta)\binom{q-1}{k-1}\right).$$
    Applying Lemma \ref{le-ent-comb} then gives a fractional perfect matching $\fpm$ of $G$ that satisfies
    \begin{align*}
        h(\fpm) &\ge \sum_{f\in E(H)}\fpmy_f h(\fpm^f)+\frac{q}{k}h(\fpmy)-\frac{n }{k }\ln\binom{n-k}{q-k} \\
        &\ge \frac{q}{k }\ln\left((p-\eta)\binom{q-1}{k-1}\right)\sum_{f\in E(H)} \fpmy_f+\frac{q}{k }\frac{n }{q}\ln\left((1-\rho)\binom{n-1}{q-1}\right)-\frac{n }{k }\ln\binom{n-k}{q-k} \\
        &\ge \frac{n }{k }\ln\left((p-\eta)(1-\rho)\binom{q-1}{k-1}\frac{\binom{n-1}{q-1}}{\binom{n-k}{q-k}}\right) \\
        &\ge \frac{n }{k }\ln\left((1-\epsilon)p\binom{n-1}{k-1}\right) \\
        &= \frac{n }{k }\ln\left((1-\epsilon)\frac{k}{n}\,\frac{\binom{n}{d}}{\binom{k}{d}}\,\delta_d(G)\right),
    \end{align*}
    where the fourth inequality holds whenever $1/\eta$ and $q$ are large enough. This completes the proof.
\end{proof}

Note that even though we are unable to prove that $c_d(k,n)/\binom{n-d}{k-d}$ tends to a limit as $n\to\infty$, the previous theorem is almost equally good in practice since an application of Theorem \ref{thm-kwan} keeps the same order of the error term, that is, $o(n)$.

%% file: approx-counting.tex
\section{Bounds for approximate counting thresholds}\label{sec-approx-count}

Recall that Theorem \ref{thm-main} gives an upper bound on counting thresholds that is only a simple extension of the bound for $1$-degree, so it does not really take $d$ into account. However, it is still possible to obtain better general bounds on $\beta_d(k)$ by relating them to those of hypergraphs with smaller uniformity. This section pursues that goal. Our ideas rely in part on the methods of Kwan, Safavi, and Wang in \cite{kwan2026counting}.

First, one can reduce the case $(d,k)$ to $(1,k/d)$ whenever $d\mid k$. The resulting bound is then $1-d/k$ instead of the $1-1/k$ bound given by Theorem \ref{thm-main}.

\begin{theorem}\label{thm-ent-div}
    For all integers $1\le d\le k-1$ such that $d\mid k$, it holds that $\beta_d(k)\le \beta_1(k/d)$.
\end{theorem}

\begin{proof}
    Fix $d$ and $k$ as in the statement of the theorem. Take arbitrary $\gamma,\epsilon>0$, and let $G$ be a $k$-uniform hypergraph on $n$ vertices with $\delta_d(G)\ge(\beta_1(k/d)+\gamma)\binom{n-d}{k-d}$, where $n$ is large enough. We need to show that
    $$h(G)\ge \frac{n}{k}\ln\left((1-\epsilon)\frac{k}{n}\,\frac{\binom{n}{d}}{\binom{k}{d}}\,\delta_d(G)\right).$$
    
    Let $H$ be an auxiliary $k/d$-uniform hypergraph on vertex set $\binom{V(G)}{d}$. Its edges are the sets $f\in\binom{V(H)}{k/d}$ such that $\bigcup f\in E(G)$, where $\bigcup f$ denotes $\bigcup_{U\in f} U$. Then for every $S\in \binom{V(G)}{d}$,
    \begin{align*}
        \deg^H_1S &=\left(\deg_d^GS\right)\frac{(k-d)!}{(k/d-1)!(d!)^{k/d-1}} \\
        &\ge (\beta_1(k/d)+\gamma)\binom{n-d}{k-d}\frac{(k-d)!}{(k/d-1)!(d!)^{k/d-1}} \\
        &\ge (\beta_1(k/d)+\gamma/2)\binom{\binom{n}{d}-1}{k/d-1},
    \end{align*}
    as long as $n$ is large enough (in terms of $\gamma$). Under the same hypothesis, now in terms of both $\gamma$ and $\epsilon$, it follows from the definition of $\beta_1(k/d)$ that there exists a fractional perfect matching $\fpmy:E(H)\to\Rpos$ of $H$ such that
    \begin{equation}\label{eq-ent-div}
        h(\fpmy)\ge \frac{\binom{n}{d}}{k/d}\ln\Big((1-\epsilon)\delta_1(H)\Big).
    \end{equation}

    For $e\in E(G)$ and $f\in E(H)$, write $e\sim f$ whenever $\bigcup f=e$, and let $a=1/\binom{n-1}{d-1}$. Define an edge-weighting $\fpm$ of $G$ by $\fpm_e:= a\sum_{e\sim f\in E(H)} \fpmy_f$. Then $\fpm$ is nonnegative. Moreover, for all $v\in V(G)$,
    \begin{align*}
        \sum_{e\ni v}\fpm_e &= a\sum_{e\ni v}\sum_{f\sim e} \fpmy_f = a\sum_{\substack{f\in E(H) \\ v\in\bigcup f}}\sum_{e\sim f}\fpmy \\
        &= a\sum_{\substack{f\in E(H) \\ v\in\bigcup f}}\fpmy_f = a\sum_{\substack{S\in\binom{V(G)}{d} \\ v\in S}} \sum_{S\in f\in E(H)} \fpmy_f \\
        &= a\sum_{\substack{S\in\binom{V(G)}{d} \\ v\in S}}1 = a\binom{n-1}{d-1} = 1,
    \end{align*}
    so $\fpm$ is a fractional perfect matching of $G$.
    
    Let $N=\frac{k!}{(k/d)!(d!)^{k/d}}$ be the number of edges $f\in E(H)$ such that $e\sim f$ for any fixed edge $e\in E(G)$. Then
    \begin{align*}
        h(G) \ge h(\fpm) &= \sum_{e\in E(G)}\Bigg( a\sum_{f\sim e}\fpmy_f\Bigg)\ln\frac{1}{aN }\frac{N }{\sum_{f\sim e}\fpmy_f} \\
        &= a\left(\ln\frac{1}{aN}\right)\sum_{f\in E(H)}\fpmy_f + aN\sum_{e\in E(G)} \frac{\sum_{f\sim e}\fpmy_f }{N}\ln\frac{N }{\sum_{f\sim e}\fpmy_f} \\
        &\ge \frac{n}{k}\ln\left(\frac{(k/d)!(d!)^{k/d}}{k!}\binom{n-1}{d-1}\right) + aN\sum_{e\in E(G)}\sum_{f\sim e}\frac{1}{N }\fpmy_f\ln\frac{1}{\fpmy_f} \\
        &= \frac{n}{k}\ln\left(\frac{(k/d)!(d!)^{k/d}}{k!}\binom{n-1}{d-1}\right) + ah(\fpmy) \\
        &\ge \frac{n}{k}\ln\left((1-\epsilon)\frac{(k/d)!(d!)^{k/d}}{k!}\binom{n-1}{d-1}\delta_1(H)\right) \\
        &= \frac{n }{k }\ln\left((1-\epsilon)\frac{(k/d)\,d!}{k\dots(k-d+1)}\binom{n-1}{d-1}\delta_d(G)\right) \\
        &= \frac{n}{k}\ln\left((1-\epsilon)\frac{k}{n}\,\frac{\binom{n}{d}}{\binom{k}{d}}\,\delta_d(G)\right),
    \end{align*}
    where the first inequality follows by Jensen's inequality applied to the concave function $x\mapsto x\ln(1/x)$, and the second one follows from (\ref{eq-ent-div}). This concludes the proof.
\end{proof}

Similarly, one can obtain a slightly better bound for $(d,k)$ by reducing it to the case $(d-l,k-l)$ for any $l<d$.

\begin{theorem}\label{thm-ent-sub}
    For all integers $1\le l < d \le k-1$, it holds that $\beta_d(k)\le \beta_{d-l}(k-l)$.
\end{theorem}

\begin{proof}
    The proof is very similar to that of Theorem \ref{thm-ent-div}. Fix $l$, $d$, and $k$ as in the statement of the theorem. Take arbitrary $\gamma,\epsilon>0$, and let $G$ be a $k$-uniform hypergraph on $n$ vertices with $\delta_d(G)\ge(\beta_{d-l}(k-l)+\gamma)\binom{n-d}{k-d}$, where $n$ is large enough. We need to show that
    $$h(G)\ge \frac{n}{k}\ln\left((1-\epsilon)\frac{k}{n}\,\frac{\binom{n}{d}}{\binom{k}{d}}\,\delta_d(G)\right).$$

    Given an $l$-set $U$ of vertices of $G$, let $H^U$ denote the $(k-l)$-uniform link hypergraph of $G$ at $U$, that is, a hypergraph with vertex set $V(G)\setminus U$ and edges all $(k-l)$-subsets $f$ such that $U\cup f\in E(G)$. For every $S\in\binom{V(H^U)}{d-l}$, we have $\deg_{d-l}^{H^U} S=\deg_d^G(U\cup S)\ge (\beta_{d-l}(k-l)+\gamma)\binom{(n-l)-(d-l)}{(k-l)-(d-l)}$. Provided that $n$ is large enough in terms of $\gamma$ and $\epsilon$, it follows from the definition of $\beta_{d-l}(k-l)$ that for every $U\in\binom{V(G)}{l}$, there exists a fractional perfect matching $\fpm^U:E(H^U)\to\Rpos$ of $H^U$ such that
    \begin{equation}\label{eq-ent-sub}
        h(\fpm^U)\ge \frac{n-l }{k-l }\ln\left((1-\epsilon)\frac{k-l }{n-l }\frac{\binom{n-l}{d-l}}{\binom{k-l}{d-l}}\delta_{d-l}(H^U)\right).
    \end{equation}

    For simplicity, given $e\in E(G)$, we write $\fpm^U_e$ for $\fpm^U_{e\setminus U}$. Let $a=\frac{k-l}{k}\frac{n }{n-l }\frac{1 }{\binom{n}{l}}$ and define an edge-weighting $\fpm$ of $G$ by $\fpm_e:=a\sum_{U\in\binom{e}{l}} \fpm^U_e$ for $e\in E(G)$. Then $\fpm$ is clearly nonnegative. Moreover, for all $v\in V(G)$,
    \begin{align*}
        \sum_{e\ni v}\fpm_e &= a\sum_{e\ni v}\sum_{U\in\binom{e}{l}}\fpm^U_e \\
        &= a\sum_{U'\in\binom{V(G)\setminus\{v\}}{l-1}}\sum_{e\supseteq U'\cup\{v\}}\fpm^{U'\cup\{v\}}_e + a\sum_{U\in \binom{V(G)\setminus\{v\}}{l}}\sum_{e\supseteq U\cup\{v\}}\fpm^U_e \\
        &= a\sum_{U'\in\binom{V(G)\setminus\{v\}}{l-1}}\frac{n-l}{k-l} + a\sum_{U\in\binom{V(G)\setminus\{v\}}{l}}1 \\
        &= a\binom{n-1}{l-1}\frac{n-l }{k-l }+a\binom{n-1}{l} \\
        &= \frac{a(n-1)!}{(l-1)!(n-l-1)!}\left(\frac{1}{k-l }+\frac{1}{l }\right) \\
        &= a\binom{n-1}{l-1}\frac{(n-l)k }{(k-l)l }=1,
    \end{align*}
    so $\fpm$ is a fractional perfect matching of $G$.

    Let $N=\binom{k}{l}$. Then
    \begin{align*}
        h(G)\ge h(\fpm)&=\sum_{e\in E(G)}\Bigg(a\sum_{U\in \binom{e}{l}}\fpm^U_e\Bigg) \ln\frac{1}{aN }\frac{N }{\sum_{U\in\binom{e}{l}}\fpm^U_e} \\
        &= a\left(\ln\frac{1}{aN}\right)\sum_{U\in\binom{V(G)}{l}}\sum_{e\supseteq U}\fpm^U_e + aN\sum_{e\in E(G)}\frac{\sum_{U\in\binom{e}{l}}\fpm^U_e }{N }\ln\frac{N }{\sum_{U\in\binom{e}{l}}\fpm^U_e } \\
        &\ge a\left(\ln\frac{1}{aN}\right)\sum_{U\in\binom{V(G)}{l}}\frac{n-l }{k-l } + aN\sum_{e\in E(G)}\sum_{U\in\binom{e}{l}}\frac{1}{N }\fpm^U_e\ln\frac{1}{\fpm^U_e } \\
        &= a\frac{n-l }{k-l }\binom{n}{l}\ln\frac{1}{aN } + a\sum_{U\in\binom{V(G)}{l}}h(\fpm^U) \\
        &\ge \frac{n }{k }\ln\frac{1}{aN} + a\sum_{U\in\binom{V(G)}{l}}\frac{n-l}{k-l}\ln\left((1-\epsilon)\frac{k-l }{n-l }\frac{\binom{n-l}{d-l}}{\binom{k-l}{d-l}}\delta_{d-l}(H^U)\right) \\
        &\ge \frac{n }{k }\ln\frac{1}{aN} + a\binom{n}{l}\cdot\frac{n-l}{k-l}\ln\left((1-\epsilon)\frac{k-l }{n-l }\frac{\binom{n-l}{d-l}}{\binom{k-l}{d-l}}\delta_d(G)\right) \\
        &= \frac{n }{k }\ln\left((1-\epsilon)\frac{k }{n}\frac{\binom{n}{l}\binom{n-l}{d-l}}{\binom{k}{l}\binom{k-l}{d-l}}\delta_d(G)\right) \\
        &= \frac{n}{k}\ln\left((1-\epsilon)\frac{k}{n}\,\frac{\binom{n}{d}}{\binom{k}{d}}\,\delta_d(G)\right),
    \end{align*}
    where the first inequality follows by Jensen's inequality applied to the concave function $x\mapsto x\ln(1/x)$, the second one follows from (\ref{eq-ent-sub}), the third one uses the fact that $\delta_{d-l}(H^U)\ge \delta_d(G)$ for every $U\in\binom{V(G)}{l}$ (with equality for at least one $U$), and the last equality follows from
    $$\frac{\binom{n}{l}\binom{n-l}{d-l}}{\binom{k}{l}\binom{k-l}{d-l}}=\frac{\binom{n}{d}}{\binom{k}{d}}.$$
    This concludes the proof.
\end{proof}

%% file: concl.tex
\section{Concluding remarks}\label{sec-concl}

In this paper we introduced (approximate) counting thresholds for $d$-degrees in $k$-uniform hypergraphs, above which a hypergraph's entropy is guaranteed to be large enough to ensure that the hypergraph contains at least as many perfect matchings as it is expected in a random hypergraph with the same edge density, i.e.\ with edge probability $p=\delta_d(G)/\binom{n-d}{k-d}$. Unlike the case $d\ge k/2$, where these thresholds are known to coincide, when $d<k/2$ they can diverge, and in this paper we initiate their study by showing some general upper bounds and a substitute for the convergence result available for Dirac thresholds.

Determining exact values of (approximate) counting thresholds could be an interesting problem for further research on the topic. For example, it follows from Theorem \ref{thm-main} that the approximate $1$-degree $3$-uniform counting threshold $\beta_1(3)$ is at most $1-1/3=2/3$. One can ask whether this is sharp, but unfortunately, we believe the answer is no. Consider Lisa Sauermann's counterexample from \cite[Section 5]{ferber2023counting}, which shows that Dirac and counting thresholds may diverge for $d<k/2$: it is a ``bipartite'' hypergraph $H_{\epsilon}$ with vertex parts $A$ and $B$ of sizes $|A|=(1/3+\epsilon)n$ and $B=(2/3-\epsilon)n$ for some small $\epsilon>0$, and edges all $3$-sets that do not have all three vertices from the same part. Calculations show that the hypergraph has the normalised minimum $1$-degree of at least $5/9$ (which is the Dirac threshold in this case), but can contain no more than $(1+o_{\epsilon}(1))^n\Phi(\multicomp 3n)\left(\frac{4}{9}\right)^{n/3}$ perfect matchings, which is smaller than the expected number in the random hypergraph.

We believe that a family of such bipartite hypergraphs can also provide a sharp lower bound for (approximate) counting thresholds, at least in the case $(d,k)=(1,3)$. In particular, calculations show that for $0<\epsilon<1/6$, the minimum $1$-degree of $H_{\epsilon}$ is approximately $\delta_1=(\frac{1}{3}+\epsilon)(\frac{5}{6}-\frac{\epsilon}{2})n^2$, while the number of perfect matchings is
\begin{align*}
    \ln\Phi(H_{\epsilon})=\ln\Phi(\multicomp 3n) + \frac{n}{3}\Bigg(&(1+3\epsilon)\ln\left(\frac{1}{3}+\epsilon\right) - (1-3\epsilon)\ln\left(\frac{1}{3}-\epsilon\right) \\
    &+ (2-3\epsilon)\ln\left(\frac{2}{3}-\epsilon\right) - 3\epsilon\ln\epsilon\Bigg) + o(n).
\end{align*}
This is larger than the desired $\ln\Phi(\multicomp 3n)+\frac{n }{3}\ln(\delta_1/\binom{n-1}{2})$ only when $\epsilon\gtrsim 0.04847$. In other words, we believe $\beta_1(3)\approx 0.61783\ll2/3$. This is also suggested by computer search.

\begin{acknowledgements}
    Parts of this paper come from the semester project that the author did as a Master student at ETH Z\"urich, Switzerland. The author would like to thank Barnab\'as Janzer, Zhihan Jin, and Benny Sudakov for their supervision of this project. The author would also like to thank Matthew Kwan, Farhood Rostamkhani, and Yiting Wang for useful comments and for spotting some mistakes.
\end{acknowledgements}

%% file: append.tex
\section{Proof of Theorem \ref{thm-pmd}}\label{sec-append}

The proof we present here is an adaptation of the original proof due to Hoffman \cite{hoffman1965nonsingularity}. Somewhat surprisingly, one single point in Theorem \ref{thm-pmd} implies all the others. This point is given by the following theorem.

\begin{theorem}\label{thm-inv-of-spmd}
    All strictly PMD matrices are invertible.
\end{theorem}

We postpone the proof and first show how Theorem \ref{thm-inv-of-spmd} implies Theorem \ref{thm-pmd}.

\begin{proof}[Proof of Theorem \ref{thm-pmd} assuming Theorem \ref{thm-inv-of-spmd}]
    First let $A$ be a PMD $n\times n$ matrix. For $\lambda>0$, $\det(A+\lambda I_n)$ is a monic polynomial in $\lambda$, so for $\lambda$ large enough, we have $\det(A+\lambda I_n)>0$. Suppose $\det A<0$. By continuity of the determinant, there is $\lambda_0>0$ such that $\det(A+\lambda_0I_n)=0$. But $A+\lambda_0I_n$ is a \emph{strictly} PMD matrix, so $\det(A+\lambda_0I_n)\neq 0$ by Theorem \ref{thm-inv-of-spmd}, a contradiction. This shows $\det A\ge0$. If $A$ is strictly PMD, we conclude again by Theorem \ref{thm-inv-of-spmd} that $\det A>0$.

    Now suppose that $A$ is an invertible PMD matrix. Then $A^{-1}=\frac{1}{\det A}C(A)^{\top}$, where $C(A)$ is the cofactor matrix of $A$. By the first part of the proof and the fact that $\det A\neq 0$, we have $\det A>0$, so that the sign of the column sums of $A^{-1}$ is equal to the sign of the row sums of $C(A)$, and it suffices to show that these are nonnegative.

    Let us denote by $A^{ij}$ the submatrix of $A$ with $i$-th row and $j$-th column removed. Let $1\le i\le n$. Then $\sum_{j=1}^n C(A)_{ij}=\sum_{j=1}^n(-1)^{i+j}\det(A^{ij})=\det A_{i\leftrightarrow\one_n}$, where $A_{i\leftrightarrow\one_n}$ is the matrix $A$ with the $i$-th row replaced by the all-ones vector $\one_n$. But $A_{i\leftrightarrow \one_n}$ is a PMD matrix as well, which implies that $\det A_{i\leftrightarrow \one_n}\ge 0$, thus concluding the proof.
\end{proof}

With some more effort, one can show that the column sums of the inverse of \emph{strictly} PMD matrices are themselves \emph{strictly} positive, and this is done by Hoffman \cite{hoffman1965nonsingularity}. However, we do not need this additional result in this paper, so we omit its proof. 

It remains to prove Theorem \ref{thm-inv-of-spmd}.

\subsection{Proof of Theorem \ref{thm-inv-of-spmd}}
    
The proof uses some basic theory of convex cones.

\begin{definition}\label{def-cone}
    Let $M$ be an $n\times m$ real matrix. The \emph{convex cone} with respect to $M$ is the subset $C(M)$ of $\R^n$ of the form $\{Mx\mid x\in\R_{\ge0}^m\}$.
\end{definition}

Equivalently, a subset $C\subseteq\R^n$ is a convex cone if for every $x\in C$ and $\alpha\ge 0$, $\alpha x\in C$ as well.

\begin{lemma}\label{le-inv-of-cons-cover}
    Let $n$ be a positive integer and let $M_1,\dots,M_n$ be $n\times m_i$, $i=1,\dots,n$, real matrices, for some positive integers $m_i$, such that the cones $C_i=C(M_i)$ and their opposites cover the whole space:
    \begin{equation}\label{eq-cons-cover}
        \bigcup_{i=1}^n (\pm C_i)=\R^n.
    \end{equation}
    Let $A$ be an $n\times n$ real matrix and suppose $A_{i*}M_i>0$ holds for all $i=1,\dots,n$, where $A_{i*}$ denotes the $i$-th row of $A$. Then $A$ is invertible.
    
\end{lemma}

\begin{proof}
    Assume that (\ref{eq-cons-cover}) holds and suppose $A$ is such that $A_{i*}M_i>0$ for all $i=1,\dots,n$. For the sake of contradiction, suppose $A$ is not invertible. Then there exists $y\in\R^n\setminus\{0\}$ such that $Ay=0$. By (\ref{eq-cons-cover}), $y\in C_i$ or $-y\in C_i$ for some $i$ and, without loss of generality, we can assume that $y\in C_i$ holds, that is, there exists $x\in\R_{\ge0}^{m_i}\setminus\{0\}$ such that $y=M_ix$. This gives, combined with the hypothesis on $A_{i*}M_i$:
    $$0=A_{i*}y=A_{i*}M_ix>0,$$
    which is the desired contradiction. 
\end{proof}

We can now proceed to the proof of Theorem \ref{thm-inv-of-spmd}. Suppose $A$ is strictly PMD. The condition (\ref{eq-pmd1}) can be restated as $A_{i*}\one_n>0$ for all $i$, and similarly, the condition (\ref{eq-pmd2}) is equivalent to $A_{i*}v_k>0$ for all $i\neq k$, where $v_i=-e_i+\frac{1}{n}\one_n$. Therefore, $A$ is strictly PMD if and only if $A_{i*}M_i>0$ holds for all $i$, where $M_i=(v_1\,\dots\,v_n)_{i\leftrightarrow \one_n}$, that is, a matrix with column vectors $v_1,\dots,v_n$, where $v_i$ was replaced by the all-ones vector $\one_n$.

We want to apply Lemma \ref{le-inv-of-cons-cover} to show that $A$ is invertible. For $M_1,\dots,M_n$ as above, $A_{i*}M_i>0$ for all $i$ is already given by the strict positive mean-dominance of $A$, so it suffices to show 
$$\bigcup_{i=1}^n\left(\pm C(M_i)\right)=\R^n.$$

\begin{claim}\label{cl-vi-independent}
    For all $i=1,\dots,n$, $\langle v_i,\one_n\rangle=0$ and any $n-1$ of the vectors $v_1,\dots,v_n$ form a basis of $\one_n^{\perp}$, the normal subspace of $\one_n$.
\end{claim}

\begin{proof}
    The first part of the claim is immediate. For the second, note that $\dim\one_n^{\perp}=n-1$, so it suffices to show that any $n-1$ of the vectors $v_1,\dots,v_n$ are linearly independent. Without loss of generality, we will show this for vectors $v_1,\dots,v_{n-1}$. So let $\alpha_1,\dots,\alpha_{n-1}\in\R$ be such that $\sum_{i=1}^{n-1} \alpha_i v_i=0$. We want to show that $\alpha_1=\dots=\alpha_{n-1}=0$.

    We have
    $$0=\sum_{i=1}^{n-1}\alpha_i v_i=\sum_{i=1}^{n-1}\alpha_ie_i+\frac{1}{n}\left(\sum_{i=1}^{n-1}\alpha_i\right)\one_n,$$
    and taking the $n$-th coordinate, we get $0=\sum_{i=1}^{n-1}\alpha_i$. But then $0=\sum_{i=1}^{n-1}\alpha_ie_i$, which implies $\alpha_i=0$ for every $i$.
\end{proof}

\begin{claim}
    \label{cl-vi-pos-combb}
    Every vector $x\in\one_n^{\perp}$ is a conical (i.e.\ nonnegative) combination of at most $n-1$ of the vectors $v_1,\dots,v_n$.
\end{claim}

\begin{proof}
    Note that
    \begin{equation}
        \label{eq-vi-sum}
        \sum_{i=1}^n v_i=\one_n-\sum_{i=1}^ne_i=0.
    \end{equation}
    
    By Claim \ref{cl-vi-independent}, $v_1,\dots,v_{n-1}$ form a basis of $\one_n^{\perp}$, so $x$ can be represented as a linear combination of $v_i$-s:
    \begin{equation}\label{eq-vi-lin-comb}
        x=\alpha_1v_1+\dots+\alpha_{n-1}v_{n-1}.
    \end{equation}
    If all $\alpha_i$-s are positive, we are done. Otherwise, for all $i$ such that $\alpha_i<0$, we can replace $v_i$ in (\ref{eq-vi-lin-comb}) by $-\sum_{j\neq i}v_j$ using (\ref{eq-vi-sum}). This gives a representation of $x$ as a conical combination of possibly all $v_i$-s:
    \begin{equation}\label{eq-conic-comb}
        x=\beta_1v_1+\dots+\beta_nv_n,\hspace{0.3cm} \beta_1,\dots,\beta_n\ge0.
    \end{equation}
    If there is $i$ with $\beta_i=0$ in (\ref{eq-conic-comb}), we are done, so suppose this is not the case and take $i$ with the minimum value of $\beta_i$. By replacing again $v_i$ by $-\sum_{j\neq i}v_j$ in (\ref{eq-conic-comb}), $v_i$-term is set to $0$, while all other coefficients remain nonnegative by minimality of $\beta_i$, so we get the desired linear combination.
\end{proof}

We will now show
$$\bigcup_{i=1}^n C_i\supseteq\{y\mid\langle y,\one_n\rangle\ge0\},$$
which will then immediately imply
$$\bigcup_{i=1}^n (\pm C_i)=\R^n,$$
thus completing the proof. To this extent, let $y\in\R^n$ be such that $\langle y,\one_n\rangle\ge0$. We want to show that there is $i\in\{1,\dots,n\}$ and $x\in\R_{\ge0}^n$ such that $y=M_ix$.

Let $y=y_0+y^{\perp}$ be the (unique) decomposition of $y$ with $y_0\parallel\one_n$ and $y^{\perp}\perp\one_n$. By Claim \ref{cl-vi-pos-combb}, $y^{\perp}$ is a conical combination of at most $n-1$ vectors of $v_1,\dots,v_n$. Without loss of generality, suppose $v_1$ is excluded from this conical representation and write $y^{\perp}=\alpha_2v_2+\dots+ \alpha_{n}v_{n}$ with $\alpha_2,\dots,\alpha_{n}\ge0$. 

We have also 
$$0\le\langle y,\one_n\rangle=\langle y_0,\one_n\rangle,$$
so by posing $\alpha_1=\langle y_0,\one_n\rangle/n$, we get $\alpha_1\ge0$ as well. Since now $\alpha_1\one_n=y_0$ and $\alpha_2v_2+\dots+\alpha_nv_n=y^{\perp}$, we get $y=M_1x$ for $x=(\alpha_1\,\dots\,\alpha_n)\ge0$. This concludes the proof of Theorem \ref{thm-inv-of-spmd}. \qed

\begin{remark}
    Theorems \ref{thm-pmd} and \ref{thm-inv-of-spmd} actually generalise to an even broader class of matrices than strictly PMD matrices. Whilst in the definition of the strictly PMD matrices one takes the unweighted mean of row entries, it is possible to take a weighted mean instead, with weights shifted circularly by $k$ positions to the right in the condition (\ref{eq-pmd2}) of Definition \ref{def-pmd}. Moreover, one can even reverse the inequality sign in (\ref{eq-pmd1}) or (\ref{eq-pmd2}), and Theorems \ref{thm-pmd} and \ref{thm-inv-of-spmd} would still hold (although with some of the promised inequalities reversed as well).
    
    Since we only need the unweighted means, we restrict our presentation here to the (strictly) PMD matrices. We refer to the paper of Hoffman \cite{hoffman1965nonsingularity} for the proof of the general case.
\end{remark}

\begin{remark}
    It has been shown by Carnicer, Goodman, and Pe\~na \cite{carnicer1999linear} that the set of conditions in Definition \ref{def-pmd} is the weakest to ensure $\det A>0$ in the sense that with any of the conditions removed, there is a $n\times n$ real matrix that satisfies the remaining conditions and whose determinant is negative.
\end{remark}